\documentclass[11pt,oneside]{amsart}

\usepackage[english]{babel}    
\usepackage[utf8]{inputenc}    
\usepackage[T1]{fontenc}       
\usepackage{amssymb} 		    
\usepackage{mathtools}         
\usepackage{icomma}            
\usepackage{units}             
\usepackage{enumerate}         
\usepackage{array}             
\usepackage[usenames]{color}   
\usepackage{graphicx}          

\usepackage{dsfont}
\usepackage{algorithm}
\usepackage{algpseudocode}

\usepackage[scale=.83]{geometry}

\usepackage[bookmarks, colorlinks=true, linkcolor=black, anchorcolor=black,
citecolor=black, filecolor=black, menucolor=black, runcolor=black,
urlcolor=black, pdfencoding=unicode]{hyperref}  

\newtheorem{thm}{Theorem}
\newtheorem{lemma}[thm]{Lemma}
\newtheorem{prop}[thm]{Propositon}
\newtheorem{cor}[thm]{Corollary}
\theoremstyle{remark}
\newtheorem{remark}[thm]{Remark}

\newcommand{\RR}{\ensuremath{\mathbb{R}}}

\renewcommand{\P}{\ensuremath{\mathbb{P}}}

\newcommand{\F}{\mathcal{F}}

\newenvironment{itemize*}{\vspace{-10pt}\begin{itemize}\setlength{\itemsep}{0pt}\setlength{\parskip}{2pt}}{\end{itemize}}
\newenvironment{enumerate*}{\vspace{-10pt}\begin{enumerate}\setlength{\itemsep}{0pt}\setlength{\parskip}{2pt}}{\end{enumerate}}
\newenvironment{description*}{\vspace{-12pt}\begin{description}\setlength{\itemsep}{0pt}\setlength{\parskip}{2pt}}{\end{description}}

\newcommand{\Sp}{\ensuremath{\mathbb{S}}}
\newcommand{\abs}[1]{\left|#1\right|}
\newcommand{\set}[2]{\left\{#1\ ; \ #2  \right\}}
\renewcommand{\d}{\mathrm{d}}

\newcommand{\e}[1]{\mathrm{e}_{#1}}

\title{A note on the exact simulation of spherical Brownian motion}

\author{Aleksandar Mijatovi{\'c}}
\address{Department of Statistics, University of Warwick, \& The Alan Turing Institute, UK}
\email{a.mijatovic@warwick.ac.uk}
\author{Veno Mramor}
\address{Department of Statistics, University of Warwick, \& The Alan Turing Institute, UK}
\email{veno.mramor@warwick.ac.uk}
\author{Ger{\'o}nimo Uribe Bravo}
\address{Instituto de Matematicas, Universidad Nacional Aut{\'o}noma de M{\'e}xico, M{\'e}xico }
\email{geronimo@matem.unam.mx}

\numberwithin{equation}{section}
\numberwithin{thm}{section}

\begin{document}

\begin{abstract}
We describe an exact simulation algorithm for the increments of Brownian motion on a sphere of arbitrary dimension,
based on the skew-product decomposition of the process with respect to the standard geodesic distance. 
The radial process is closely related to a Wright-Fisher diffusion, increments of which can be simulated exactly
using the recent work of Jenkins \& Span\`{o}~(2017)~\cite{exsimWF}. The rapid spinning phenomenon of the skew-product decomposition
then yields the algorithm for the increments of the process on the sphere. 

\smallskip
\noindent \textsc{Keywords.}
exact simulation; skew-product decomposition; spherical Brownian motion; Wright-Fisher diffusion
\end{abstract}

\maketitle

\section{Introduction}
Brownian motion (BM) can be viewed as a continuous time random walk with symmetric increments
and is as such of fundamental importance in physics and other natural sciences.
In applications
one is often interested in the BM on curved surfaces
and other manifolds, 
see e.g.~\cite{KrishnaFluor} and~\cite{BactLi1} for the modelling of  
the fluorescent marker molecules in cell membranes  and the motion of
bacteria or any other diffusing particles, respectively. 
Often  Monte Carlo simulation algorithms for such models on curved spaces are constructed using approximate tangent plane methods, 
which are accurate only for very small time steps, making the algorithms computationally expensive. 
Algorithms allowing simulation over longer time steps are hence of particular interest in the physics literature.  
For example,~\cite{sim4dimsphere} gives
a simulation algorithm for the BM on the three-dimensional sphere, based on its quaternionic structure. 
This algorithm is neither exact nor does it generalise easily to other dimensions. In~\cite{sim2dimsphere} 
it is applied to design an approximate simulation algorithm for the BM on the two-dimensional sphere.
A further approximate algorithm for the simulation of BM on the two-sphere is given in~\cite{gaussForSph}.
All the aforementioned algorithms are based on an approximation of the transition density of the BM on the relevant sphere.

In contrast, 
for any dimension $d\ge3$,
Algorithm~\ref{simSpB} 
simulates exactly the increments of the BM $Z$ on the 
sphere $\Sp^{d-1}:=\set{z\in \RR^{d}}{\abs{z}=1}$.  
It is based on two facts established in Section~\ref{moreRes}: the radial part of 
$Z$ (with respect to the standard metric on $\Sp^{d-1}$) can be transformed to a Wright-Fisher diffusions 
and, due to the rapid spinning of the skew-product decomposition of $Z$, its angular component is uniform on $\Sp^{d-2}$. 

\begin{algorithm}\caption{Simulation of the increment of Brownian motion  $Z$ on $\Sp^{d-1}$ over any time interval} \label{simSpB}
\begin{algorithmic}[1]
	\Require{Starting point $z\in\Sp^{d-1}$ and time horizon $t>0$} 
	\State Simulate the radial component: $X\sim \mathrm{WF}_{0,t}\left(\frac{d-1}{2},\frac{d-1}{2}\right)$ \Comment Algorithm~\ref{simWF} in Appendix~\ref{appendix:alg}
	\State Simulate the angular component: $Y$ uniform on $\Sp^{d-2}$  \label{step:uniform_sphere}                    
	\State Set $u:=(\e{d}-z)/\abs{\e{d}-z}$ and $O(z):=I-2uu^\top$ \label{stepOrt}
	\State \textbf{return} $O(z)(2\sqrt{X(1-X)}Y^\top,1-2X)^\top$
\end{algorithmic}
\end{algorithm}

The vectors in Algorithm~\ref{simSpB} are column vectors of appropriate dimension,
$\e{d}:=(0,\ldots,0,1)^\top\in\Sp^{d-1}$ denotes the north pole of the sphere and 
$O(z)$ is the reflection of $\RR^d$ across the hyperplane through the origin with the normal
$u$. 
Algorithm~\ref{simWF} 
is the exact simulation algorithm for the increment of the Wright-Fisher diffusion given in~\cite[Alg.~1]{exsimWF} 
(see Eq.~\eqref{def:WFdiff} below for  a definition of Wright-Fisher diffusions). 
Step~\ref{step:uniform_sphere} in Algorithm~\ref{simSpB} consists of simulating a vector $N$ in $\RR^{d-1}$
with independent standard normal components and setting $Y=N/\abs{N}$ (we denote the standard
Euclidean norm by $\abs{\cdot}$ throughout).

The key property of the orthogonal matrix $O(z)\in \RR^d\otimes\RR^d$ in Algorithm~\ref{simSpB}
is $O(z)\e{d}=z$. 
In fact, any orthogonal matrix in $\RR^d\otimes\RR^d$ with this property would lead to an exact sample 
from the increment of BM on $\Sp^{d-1}$.
Indeed, if 
$O_1(z),O_2(z)\in\RR^d\otimes\RR^d$ are two such orthogonal matrices,  then the product 
$O_1^{-1}(z)O_2(z)$
fixes 
$\e{d}$
and is given by an orthogonal transformation 
$\widetilde{O}(z)\in\RR^{d-1}\otimes\RR^{d-1}$ on the orthogonal complement $\{\e{d}\}^\perp$ in $\RR^d$. 
Hence
$O_2(z)(2\sqrt{X(1-X)}Y^\top,1-2X)^\top=O_1(z)(2\sqrt{X(1-X)}(\widetilde{O}(z)Y)^\top,1-2X)^\top$ 
implying
$$O_2(z)(2\sqrt{X(1-X)}Y^\top,1-2X)^\top\overset{d}{=}O_1(z)(2\sqrt{X(1-X)}Y^\top,1-2X)^\top,$$ 
where 
$\overset{d}{=}$ denotes equality in law. 
The formula for $O(z)$ in Algorithm~\ref{simSpB} is chosen due to
its simplicity.

Algorithm~\ref{simSpB} exploits the symmetry of both the geometry of $\Sp^{d-1}$ and the law of spherical BM
to reduce the simulation problem in any dimension to the simulation of an increment of a one-dimensional
diffusion. 
Unlike the simulation methods in~\cite{sim4dimsphere,sim2dimsphere,gaussForSph},
Algorithm~\ref{simSpB} 
depends on the dimension $d$ only through the value of the mutation parameters in the Wright-Fisher diffusion. 
As discussed in~\cite[Sec.~4]{exsimWF}, Algorithm~\ref{simWF}, which simulates exactly from the law of the increment of the
Wright-Fisher diffusion, is numerically stable for the time intervals of length  $t\geq0.05$. 
This is essentially because, 
when $t$ is very small, the modulus of the summands in the alternating series~\cite[Eq.~(5)]{exsimWF}
is increasing as a function of the index for large values of the index, before it starts to decrease monotonically. 
Conversely, Algorithm~\ref{simWF} (and hence Algorithm~\ref{simSpB}) becomes more efficient with increasing time horizon $t$ 
(for small $t$, a normal approximation can be used to obtain a fast approximate algorithm substituting Algorithm~\ref{simWF},
see~\cite[Thm~1]{exsimWF} and the comments therein). 
Dimension $d$ does not affect the performance of Algorithm~\ref{simSpB}, cf. Appendix~\ref{appendix:alg}.

Algorithm~\ref{simSpB} above 
and the Markov property of the BM $Z$ on the sphere $\Sp^{d-1}$ 
yield an exact sample from any finite-dimensional marginal of $Z$.  Moreover, as the
BM on $\Sp^1\subseteq \mathbb{C}$ can be represented as 
$Z=\exp(\mathrm{i} W)$, where $\mathrm{i}^2=-1$  and $W$ denotes a standard BM
on $\RR$, the assumption $d\geq3$ 
in Algorithm~\ref{simSpB} is not restrictive.

Our main technical result  
(Proposition~\ref{skewProd} below)
states that the spherical BM $Z$ 
enjoys a skew-product decomposition in which the last component of $Z$ is independent of
the normalized and suitably time-changed remaining components.
Its main application  in the present paper
consists of establishing the validity of Algorithm~\ref{simSpB}.
The proof of Proposition~\ref{skewProd} is analogous to that of~\cite[Thm~1.5]{projSphere}.
However, we note that Proposition~\ref{skewProd} is not 
a corollary of~\cite[Thm~1.5]{projSphere}, as the quoted theorem implies only 
the independence of the modulus of the last component (but not of its sign).  
This difference makes the claim of
Proposition~\ref{skewProd} more general than~\cite[Thm~1.5]{projSphere} and (we hope) of independent interest. 

The classical skew-product decomposition of Brownian motion $\zeta$ in $\RR^d$,
started away from the origin, features the BM $Z$ on $\Sp^{d-1}$ as follows:
$Z_t=\zeta_{C_t}/|\zeta_{C_t}|$, where the increasing process $(C_t)_{t\geq0}$
is the inverse of the additive functional 
$\int_0^\cdot |\zeta_u|^{-2}\d u$ of $\zeta$. Moreover, the process $Z$ and
$|\zeta|$ are independent of each other.  It is tempting to try to use the
above representation to obtain a sample from $Z_t$. One would first need to
simulate $C_t$ , then get a sample from the conditional law of $\zeta_{C_t}$,
given $C_t$, and finally normalise this sample to get an
$\Sp^{d-1}$-valued random element identically distributed as $Z_t$.  
We stress however that this does not yield an approximate simulation algorithm
for the increments of $Z$, let alone an exact one. Even if one is able to
simulate $C_t$, which is in itself non-trivial and can to the best of our knowledge
only be done approximately, the key obstacle is the fact that the process
$\zeta$ is not independent of the time $C_t$.  More precisely, since $C_t$
depends 
on the trajectory of $\zeta$ up to $C_t$, the conditional law of $\zeta_{C_t}$,
given $C_t$, is intractable, so the second step of the proposed sampling
procedure cannot be achieved.

Finally, we note that Algorithm~\ref{simSpB} yields an algorithm for the exact simulation of the increments of 
the BM on the real 
$\mathbb{R}\mathbb{P}^n$, complex $\mathbb{C}\mathbb{P}^n$ and quaternionic $\mathbb{H}\mathbb{P}^n$ projective spaces
(for any integer $n$). 
These are Riemannian manifolds of the (real) dimension $n$, $2n$ and $4n$, respectively, with canonical Riemannian metrics
described in Appendix~\ref{BMproj}. The Riemannian submersion $\pi$ (mapping $\Sp^n\to\mathbb{R}\mathbb{P}^n$,
$\Sp^{2n+1}\to\mathbb{C}\mathbb{P}^n$ and
$\Sp^{4n+3}\to\mathbb{H}\mathbb{P}^n$) by Lemma~\ref{lem:FPnBM} below projects the BM $Z$ on the relevant sphere 
to the BM $\pi(Z)$ on the projective space. Since $\pi(z)=[z]$ only converts the standard coordinates of the point $z$ on the sphere 
to the homogeneous coordinates $[z]$ in the projective space,  the random element 
$$\pi\left(O(z)(2\sqrt{X(1-X)}Y^\top,1-2X)^\top\right),$$
where matrix $O(z)$ and random variables $X,Y$ are the same as in Algorithm~\ref{simSpB}, gives the exact sample of the increment of the BM on the projective space, started at $[z]$, over the time horizon $t>0$.

\section{The skew-product decomposition and Algorithm~\ref{simSpB}} \label{moreRes}
Brownian motion $Z=(Z_t)_{t\geq0}$
on the sphere
$\Sp^{d-1}$ 
is a Feller process generated by the Laplace-Beltrami operator corresponding to the
Riemannian metric on 
$\Sp^{d-1}$  
induced by the ambient Euclidean space $\RR^d$, see~\cite{hsumanifolds}.
There are a number of  different ways of representing BM on a sphere.
The most useful for our purposes is
the Stroock representation of $Z$, given (in It\^o form) by SDE~\eqref{BMsphere} 
on $\RR^d$,
\begin{equation} \label{BMsphere}
\d Z_t=(I-Z_tZ_t^\top) \d B_t-\frac{d-1}{2}Z_t \d t,\qquad Z_0\in \Sp^{d-1}, 
\end{equation}
which possesses the unique strong solution,
where 
$B$ is a BM on $\mathbb{R}^{d}$ and $I$ denotes the identity matrix of appropriate dimension (cf. \cite[Ch.3, \S3, p.83]{hsumanifolds}). 
However, as~\eqref{BMsphere} and other representations of $Z$  alluded to above are non-constructive, 
we cannot use them directly for the exact simulation of the increments of $Z$ in dimension $d\geq3$. 

As explained in the introduction, 
the key idea is to identify the skew-product decomposition of $Z$
and exploit the rapid spinning phenomenon of the angular component at the starting point $Z_0$, together with
the symmetries of the sphere $\Sp^{d-1}$ and the law of the spherical BM $Z$ to obtain Algorithm~\ref{simSpB}.
Let $D=(D_t)_{t\ge 0}$ be the geodesic distance  $D_t:=\operatorname{dis}(Z_t)\in[0,\pi]$
between $Z_t$ and some fixed point $w\in \Sp^{d-1}$. 
In~\cite[p.~269]{ItoMcKean} and~\cite{PauRogSkew}, the authors show that $D$ satisfies the SDE 
$\d D_t=\d \beta_t+\frac{d-2}{2}\cot(D_t)\d t$, 
where $\beta$ is a scalar BM. 
Since $\operatorname{dis}(z)=\arccos(\langle z,w \rangle)$ for any $z,w\in\Sp^{d-1}$, the natural transformation
$\widetilde{\mathcal{X}}:=\cos(D)$
leads to the Jacobi diffusion (see e.g.~\cite{WarrenYorJacobi} 
and the proof of Proposition~\ref{skewProd} below)
satisfying 
the SDE 
\begin{equation} \label{eqn:tildeX}
	\d \widetilde{\mathcal{X}}_t=\sqrt{1-\widetilde{\mathcal{X}}^2_t}\d \widetilde \beta_t-\frac{d-1}{2}\widetilde{\mathcal{X}}_t\d t,\qquad \widetilde{\mathcal{X}}_0=\langle Z_0,w\rangle,
\end{equation}
where $\widetilde \beta=-\beta$ is a standard scalar BM. 
This simple (and in our context crucial) observation has
to the best of our knowledge not been made in the literature so far. 
Note that 
$\widetilde{\mathcal{X}}_t=\langle Z_t,w \rangle$. We henceforth fix  $w=\e{d}$ and obtain $\widetilde{\mathcal{X}}_t=Z^d_t$, 
making $\widetilde{\mathcal{X}}$ the  process considered in~\cite[Proposition 1.1]{projSphere} for $n=1$ and $\ell=d-1$. 
A linear transformation 
$\mathcal{X}_t:=\frac{1-\widetilde{\mathcal{X}}_t}{2}$ 
yields the SDE 
\begin{equation} \label{WFproj}
\d \mathcal{X}_t=\sqrt{\mathcal{X}_t(1-\mathcal{X}_t)}\d  \beta_t+\left(\frac{d-1}{4}(1-\mathcal{X}_t)-\frac{d-1}{4}\mathcal{X}_t\right)\d t,\qquad \mathcal{X}_0=\frac{1-Z_0^d}{2},
\end{equation} 
making $\mathcal{X}$ 
a Wright-Fisher diffusion with mutation parameters $\theta_1=\frac{d-1}{2}=\theta_2$, see~\eqref{def:WFdiff} below.

A weak form of the skew-product decomposition (on the level of generators) of the BM on the sphere has been established 
in~\cite{PauRogSkew}, see also~\cite[p.~269]{ItoMcKean}. 
In order to give a path-wise skew-product representation of $Z$, note first  that for $d\ge3$ the Wright-Fisher diffusion $\mathcal{X}$ visits
neither $0$ nor $1$. We may thus introduce the following time-changes:  for
$0\le s \le t$ define $S_s(t):=
\int_s^t\frac{1}{4\mathcal{X}_u(1-\mathcal{X}_u)}\d u$, satisfying
$\lim\limits_{t\to \infty}S_s(t)=\infty$ (see proof of Proposition~\ref{skewProd} below), 
and its inverse $T_s\colon [0,\infty) \to [s,\infty)$.  
We now state our main result.

\begin{prop}[Skew-product decomposition of the BM on $\Sp^{d-1}$]\label{skewProd}
Let $d\ge3$ and $Z$ be a solution of SDE~\eqref{BMsphere}. 
Pick $s\in[0,\infty)$
and assume
that either $s>0$ or $s=0$ and $Z_0\neq \e{d}$. Let $U=(Z^1,\ldots,Z^{d-1})^\top$ denote the first $d-1$ components of $Z$. Then 
$\abs{U}=2\sqrt{\mathcal{X}(1-\mathcal{X})}$ 
and the process
$\widehat{V}=(\widehat{V}_t)_{t\geq0}$, given by $\widehat{V}_t:=U_{T_s(t)}/\abs{U_{T_s(t)}}$ 
is a BM on the sphere $\Sp^{d-2}$, started at $\widehat{V}_0=U_s/\abs{U_s}$, independent of $\mathcal{X}$. 
Hence we obtain the skew-product decomposition
$Z_t=(2\sqrt{\mathcal{X}_t(1-\mathcal{X}_t)}\widehat{V}_{S_s(t)}^\top,1-2\mathcal{X}_t)^\top$ for $t\ge s$. Furthermore,
if $Z_0 = \e{d}$ (i.e. $U_0=0$), then $\widehat{V}_t$ is uniformly distributed on $\Sp^{d-2}$ for
any $t>0$ and subsequently evolves as a stationary BM on the sphere. 
\end{prop}

Since time changing a stationary process by an independent time-change does not
affect the marginal distributions of the process, the second part of the theorem immediately yields 
Corollary~\ref{corSim}, which in turn justifies Algorithm~\ref{simSpB} for the exact simulation 
of the increments of  BM on the sphere $\Sp^{d-1}$.

\begin{cor} \label{corSim}
Let $d\ge3$ and $Z$ be the BM on $\Sp^{d-1}$ started at 
$Z_0 = \e{d}$. Additionally, let $\mathcal{W}$ be the unique strong solution of
SDE~\eqref{def:WFdiff} with mutation parameters $\theta_1=\frac{d-1}{2}=\theta_2$, started at
$\mathcal{W}_0=0$,
and $Y$ a uniformly distributed random vector on $\Sp^{d-2}$,
independent of $\mathcal{W}$. Then for each $t\ge0$, the  random vectors $Z_t$ and
$(2\sqrt{\mathcal{W}_t(1-\mathcal{W}_t)}Y^\top,1-2\mathcal{W}_t)^\top$ are identically distributed.  
\end{cor} 

It may appear that the skew-product decomposition in
Proposition~\ref{skewProd} can be applied directly to simulate the increments of the spherical BM 
started at any point in $\Sp^{d-1}$. Since the increment of the Wright-Fisher diffusion can be simulated
exactly using Algorithm~\ref{simWF}  below,
Proposition~\ref{skewProd}  
appears to reduce the problem to the simulation of the increment of
the BM on $\Sp^{d-2}$.
Recursively,
this would reduce the problem 
to the simulation of the BM 
on $\Sp^1$, where the algorithm is trivial.  
Unfortunately, for this
direct approach to work we would need to obtain a sample from  law of the pair
$(\mathcal{X}_t,\int_{0}^{t}\frac{1}{4\mathcal{X}_u(1-\mathcal{X}_u)}\d u)$, 
which at the time of writing we do not know how to do. 
As stated in Corollary~\ref{corSim},
this problem disappears when the BM $Z$ starts at the north pole $\e{d}$. 
Since for any orthogonal matrix $A\in \mathbb{R}^d\otimes\mathbb{R}^d$
the process $AZ$ solves the SDE in~\eqref{BMsphere} started at $AZ_0$, 
Corollary~\ref{corSim} 
and the orthogonal matrix $O(z)$ in Algorithm~\ref{simSpB}  
circumvent the need to simulate from the
law of the pair 
$(\mathcal{X}_t,\int_{0}^{t}\frac{1}{4\mathcal{X}_u(1-\mathcal{X}_u)}\d u)$ 
for an arbitrary starting point $z\in\Sp^{d-1}$.

\begin{remark}
It is not difficult to see that 
the Brownian motion $\widetilde Z=(\widetilde Z_t)_{t\geq0}$ on the sphere of radius $R>0$ in $\mathbb{R}^d$ satisfies the following SDE with a unique strong solution
(see e.g.~\cite[Lem.~3.6(e)]{elldiff}):
\begin{equation*} 
\d \widetilde Z_t=(I-\widetilde Z_t\widetilde Z_t^\top/|\widetilde Z_t|^2) \d B_t-\frac{d-1}{2}\widetilde Z_t/|\widetilde Z_t|^2 \d t,\qquad |\widetilde Z_0|=R. 
\end{equation*}
Thus we may define the BM $\widetilde Z$ by  $\widetilde Z_t=RZ_{t/R^2}$, where $Z$ is the BM on $\Sp^{d-1}$ satisfying the SDE in~\eqref{BMsphere},
and Algorithm~\ref{simSpB} can be applied to produce the exact sample of the increments of $\widetilde Z$. 
\end{remark}

\begin{proof}[Proof of Proposition~\ref{skewProd}]
In order to apply~\cite[Thm~1.5]{projSphere},
let $n=d-1$, $\ell=1$ and set $\gamma\equiv 1$, $g\equiv \frac{d-1}{2}$. Then the
process $U$ in Proposition~\ref{skewProd} equals the process $X$
considered in~\cite[Thm~1.5]{projSphere}. In particular, 
the processes $\abs{U}$ and $R$ are equal implying that the 
time-changes $S_s(t)$ and $T_s(t)$ used in both Proposition~\ref{skewProd}
and \cite[Thm~1.5]{projSphere} coincide.  Moreover, since \cite[Lem.~2.3]{projSphere} holds in the current setting the time-changes are 
well-defined. We finally
note that $\abs{U}=2\sqrt{\mathcal{X}(1-\mathcal{X})}.$

As mentioned in the introduction, the only thing left to prove is the
independence of $\widehat{V}$ and $\mathcal{X}=(1-Z^d)/2$. The argument is analogous to the 
one that yielded 
the independence of
$\widehat{V}$ and $R$ in the proof of \cite[Thm~1.5]{projSphere}, but the statement does not follow directly from~\cite[Thm~1.5]{projSphere}.  
Let $V:=U/\abs{U}$ and note that, 
as in the proof of~\cite[Thm~1.5]{projSphere},
It\^o's formula applied to $V$
yields the following
$\d V_t= \abs{U_t}^{-1}(I-V_{t}V^\top_{t})\d
\widetilde{B}_t -{\abs{U_t}^{-2}}\frac{d-2}{2}V_{t}\d t,$ where
$\widetilde{B}=(B^1,\ldots,B^{d-1})^\top$. 
By~\cite[Prop.~1.1]{projSphere} (with $n=1, \ell=d-1$) we have that 
$\widetilde{\mathcal{X}}=Z^d$ satisfies SDE~\eqref{eqn:tildeX} above with
$\widetilde{\mathcal{X}}_0=Z^d_0$. Moreover, as in the proof of~\cite[Prop.~1.1]{projSphere}, the scalar BM $\widetilde{\beta}$ in~\eqref{eqn:tildeX} is given by
$\d \widetilde{\beta}_t=-(1-2\mathcal{X}_t)V^\top_t\d
\widetilde{B}_t+\abs{U_t}\d B_t^d$. Consequently, $\mathcal{X}$ satisfies
SDE~\eqref{WFproj} with $\mathcal{X}_0=\frac{1-Z^d_0}{2}$ and the scalar BM
$\beta=-\widetilde{\beta}$. 
Since
$\widehat V=V_{T_s(\cdot)}$,
the change-of-time formulae~\cite[Ch.~V, \S1]{revuzyor} for the It\^o and the
Lebesgue-Stieltjes integrals imply that
$\widehat{V}_t-\widehat{V}_0=\int_0^t(I-\widehat{V}_u\widehat{V}_u^\top)\d
\widetilde{W}_u-\int_0^t\frac{d-2}{2}\widehat{V}_u\d u$, 
where 
$\widetilde{W}_t =  \int_s^{T_s(t)} \abs{U_u}^{-1}\d \widetilde{B}_u$ is a standard BM on $\mathbb{R}^{d-1}$ by Levy's
characterization. This implies that $\widehat{V}$ solves the SDE in~\eqref{BMsphere} in dimension $d-1$,
making it a BM on
$\Sp^{d-2}$. 

In order to 
conclude that $\mathcal{X}$ and
$\widehat{V}$ are independent, 
we assume without loss of generality that the underlying probability space with filtration $(\mathcal{F}_t)_{t\geq0}$
supports a further scalar $(\mathcal{F}_t)$-BM 
$\xi$, independent of $B$. 
Define the continuous local martingale $W=(W_t)_{t\geq0}$ in $\mathbb{R}^{d-1}$ 
by
$$W_t =  \int_s^{T_s(t)} \frac{1}{\abs{U_u}}(I-V_{u}V^\top_{u})\d \widetilde{B}_u
+\int_s^{T_s(t)}\frac{V_u}{\abs{U_u}} \d \xi_u.$$ 
Since 
$(I-V_{u}V^\top_{u})(I-V_{u}V^\top_{u})=I-V_{u}V^\top_{u}$,
it follows easily from the independence of $B$ and $\xi$ that 
the components of its quadratic covariation $W$ are given by
$\langle W^i,W^j \rangle_t=\int_s^{T_s(t)}\delta_{ij}{\abs{U_u}^{-2}}\d u=\delta_{ij}t$
for any $i,j\in\{1,\ldots,d-1\}$.
Hence, by Levy's characterisation, the process
$W$ is a standard $(\mathcal{G}_t)$-BM, where the filtration $(\mathcal{G}_t)_{t\geq0}$ is defined by  
$\mathcal{G}_t:=\F_{T_s(t)}$ for all $t\geq0$. 
Moreover, since
$$(I-V_uV_u^\top)\cdot \begin{bmatrix}\frac{1}{\abs{U_u}}(I-V_{u}V^\top_{u}) ,&
\frac{V_u}{\abs{U_u}}
\end{bmatrix}=\begin{bmatrix}\frac{1}{\abs{U_u}}(I-V_{u}V^\top_{u}),& 0
\end{bmatrix}$$ 
we can apply the change-of-time formulae for the stochastic and
Lebesgue-Stieltjes integral~\cite[Ch.~V, \S1]{revuzyor} to obtain 
$\widehat{V}_t-\widehat{V}_0=\int_0^t(I-\widehat{V}_u\widehat{V}_u^\top)\d W_u-\int_0^t\frac{d-2}{2}\widehat{V}_u\d u$.

Since SDE in~\eqref{BMsphere} in dimension $d-1$
has a unique strong solution,
the independence of $\widehat{V}$ 
and $\mathcal{X}$ 
follows from the independence of the driving BMs $W$ and $\beta$. 
Since $W$ and $\beta$ 
run on different time scales, we first note that the 
Markov property of $W$ implies that $W$
depends on $\mathcal{G}_0=\F_s$ only
through its position at time zero, i.e. $W_0=0$, making it independent of $\F_s$. 
In particular,
$W$ is independent of
$(\beta_{t})_{t\in [0,s]}$. Therefore it is sufficient to prove that $W$ is
independent of $(\beta_{u+s}-\beta_s)_{u\ge 0}$. For any $t\in[0,\infty)$, define
$\eta_t:=\beta_{T_s(t)}-\beta_s =\int_s^{T_s(t)} (1-2\mathcal{X}_u)V^\top_u\d \widetilde{B}_u-\int_s^{T_s(t)} \abs{U_u}\d B^d_u$ so that
$\eta$ is a $(\mathcal{G}_t)$-local martingale. Simple calculations show that
$\langle W^i,\eta \rangle_t=0$ for any component $W^i$ of $W$ and the quadratic variation $\langle \eta\rangle_t=T_s(t)-s$ satisfies
$\langle\eta\rangle_{S_s(u+ s)}=u$ for all $u\geq0$.  By Knight's Theorem (also known as the multidimensional
Dambis-Dubins-Schwarz Theorem~\cite[Ch.~V, Thm~1.9]{revuzyor}) we have 
that $W$ and $(\eta_{S_s(t+s)})_{t\ge 0}=(\beta_{s+t}-\beta_s)_{t \ge 0}$ are
independent BMs. This concludes the proof of the proposition. 
\end{proof}

\appendix 
\section{Algorithms for simulation of Wright-Fisher diffusion}
\label{appendix:alg} 
We recall the algorithm 
for the exact simulation of the marginals of Wright-Fisher diffusions.
This appendix is taken from~\cite[Section 2]{exsimWF}.
Fix  parameters $\theta_1,\theta_2\geq0$ 
and consider 
the SDE 
\begin{equation} \label{def:WFdiff}
\d
\mathcal{W}_t=\sqrt{\mathcal{W}_t(1-\mathcal{W}_t)}\d B_t +
\frac{1}{2}(\theta_1(1-\mathcal{W}_t)-\theta_2 \mathcal{W}_t)\d t, \qquad \mathcal{W}_0=x\in[0,1].
\end{equation} 
Its unique strong solution is know as the 
Wright-Fisher diffusion with mutation parameters $\theta_1$ and $\theta_2$.\footnote{This 
notation coincides with the one used in \cite{exsimWF} but differs from the one in~\cite{projSphere} by a factor of $2$.} 

\begin{algorithm}\caption{Exact simulation from the law $\mathrm{WF}_{x,t}(\theta_1,\theta_2)$}  \label{simWF}
	\begin{algorithmic}[1]
		\Require Mutation parameters $\theta_1$ and $\theta_2$, starting point $x\in[0,1]$ and time horizon $t>0$
		\State Simulate $M\overset{d}{=} A^\theta_\infty(t)$ \label{stepAinf} \Comment Use \cite[Alg.~2]{exsimWF}
		\State Simulate $L\sim \mathrm{Binomial}(M,x)$
		\State Simulate $Y\sim \mathrm{Beta}(\theta_1+L,\theta_2+M-L)$
		\State \textbf{return} $Y$
	\end{algorithmic}
\end{algorithm} 

The random variable $A^\theta_\infty(t)$ in step~\ref{stepAinf} is
integer-valued with the mass function
$\{q^\theta_m(t);m=0,1,\ldots\}$, where $\theta=\theta_1+\theta_2$, 
that can be described as follows.
Let $\{A^\theta_n(t);t\ge0\}$ be a pure
death process on the non-negative integers, started at $A^\theta_n(0)=n$, where the only transitions are of the form
$m\mapsto m-1$ and occur at rate $m(m+\theta-1)/2$ for each $m\in\{1,\ldots, n\}$. 
Then $q^\theta_m(t)=\P[A^\theta_\infty(t)=m]$ can be expressed as the limit 
$q^\theta_m(t)=\lim_{n\to \infty}\P[A^\theta_n(t)=m]$.  
The coefficients $q^\theta_m(t)$ are known in terms of an  alternating
series~\cite[Eq.~(5)]{exsimWF}, with summands whose absolute values are ultimately monotonically decreasing. 
This property is exploited in \cite[Alg.~2]{exsimWF} for the simulation of $A^\theta_\infty(t)$.
In theory, \cite[Alg.~2]{exsimWF} is exact for all
values of parameters. In practice, 
when the time horizon $t$ is very small, the algorithm runs into the  
running time and 
floating number precision problems, as it  
requires the  multiplication of very
large and very small numbers. 
It is noted in~\cite{exsimWF} that 
\cite[Alg.~2]{exsimWF} works well 
for  $t\geq0.05$ (the values of $\theta_1,\theta_2$ do not seem to
affect the performance). For $t<0.05$, a normal approximation
for $A^\theta_\infty(t)$ can be used, see~\cite[Thm~1]{exsimWF}. This yields an efficient 
approximate simulation algorithm for the law $\mathrm{WF}_{x,t}(\theta_1,\theta_2)$.

\section{Brownian motion on projective spaces}\label{BMproj}
Pick a field $\mathbb{F}\in\{\mathbb{R},\mathbb{C},\mathbb{H}\}$, where $\mathbb{C}$ are the complex numbers and $\mathbb{H}$ denotes the quaternions, 
consider $\mathbb{F}^{n+1}$ as an $(n+1)$-dimensional vector space over $\mathbb{F}$
and recall that the $n$-dimensional projective space $\mathbb{F}\mathbb{P}^n$ is defined as a space
of all $1$-dimensional subspaces of $\mathbb{F}^{n+1}$. 
More precisely, define  $\mathbb{F}\mathbb{P}^n$ to be the
set of all equivalence classes  $[x_0:\cdots:x_n]$, 
where the $(n+1)$-tuple $(x_0,\ldots,x_n)^\top$ is in $\mathbb{F}^{n+1}\backslash\{0\}$ and 
$[x_0:\cdots:x_n]=[y_0:\cdots:y_n]$ if and only if there exists a scalar
$\lambda\in\mathbb{F}\backslash\{0\}$ such that $x_i=\lambda y_i$ for each $i\in\{0,\ldots,n\}$. 
For any $x\in\mathbb{F}^{n+1}\backslash\{0\}$, 
$[x]:=[x_0:\cdots:x_n]$ denotes  the \textit{homogeneous coordinates} 
of the corresponding point in $\mathbb{F}\mathbb{P}^n$.

Let
$\pi\colon \mathbb{F}^{n+1}\backslash\{0\}\to \mathbb{F}\mathbb{P}^n$ be 
the quotient map given by $\pi(x):=[x]$.
$\mathbb{F}\mathbb{P}^n$ 
is topologised 
by the quotient topology i.e. a subset
$U\subseteq\mathbb{F}\mathbb{P}^n$ is open if and only if $\pi^{-1}(U)$ is open in
$\mathbb{F}^{n+1}\backslash\{0\}$. 
It is easy to see that $\mathbb{F}\mathbb{P}^n$ is 
Hausdorff, second-countable and compact since the restriction of $\pi$ 
to the sphere 
$\Sp^{n}$ (resp. $\Sp^{2n+1}$, $\Sp^{4n+3}$) if
$\mathbb{F}=\mathbb{R}$ (resp. $\mathbb{C}$, $\mathbb{H}$),
is surjective. 
Moreover, $\mathbb{F}\mathbb{P}^n$ is a  smooth manifold. Indeed, for any $i\in\{0,\ldots,n\}$,  let 
$\widetilde{U_i}=\set{x\in\mathbb{F}^{n+1}}{x_i\neq0}$ and note that
$\widetilde{U_i}=\pi^{-1}(\pi(\widetilde{U_i}))$, making the  set
$U_i=\pi(\widetilde{U_i})$ open in $\mathbb{F}\mathbb{P}^n$. Define the
chart $\varphi_i\colon U_i\to\mathbb{F}^{n}$  by 
$\varphi_i([x])=(\frac{x_0}{x_i},\ldots,\frac{x_{i-1}}{x_i},\frac{x_{i+1}}{x_i},\ldots,\frac{x_n}{x_i})^\top$. 
The map $\varphi_i$ is well-defined (since it assigns the same value to each $(n+1)$-tuple in a given equivalence class in $U_i$) 
and a homeomorphism 
with the inverse 
$\varphi_i^{-1}(y_1,\ldots,y_n)=[y_1:\cdots:y_{i-1}:1:y_i:\cdots:y_n]$. 
The charts $(U_i,\varphi_i)$, $i\in\{0,\ldots,n\}$ are smoothly compatible. 
Thus
$\mathbb{F}\mathbb{P}^n$ is a compact smooth manifold of dimension $n$ (resp. $2n$, $4n$) if
$\mathbb{F}=\mathbb{R}$ (resp. $\mathbb{C}$, $\mathbb{H}$). 

An equivalent way of defining the smooth structure on $\mathbb{F}\mathbb{P}^n$
is via the action 
on $\Sp^{n}$ (resp. $\Sp^{2n+1}$, $\Sp^{4n+3}$) if
$\mathbb{F}=\mathbb{R}$ (resp. $\mathbb{C}$, $\mathbb{H}$)
of the multiplicative group of unit elements in $\mathbb{F}$,
see proof of Lemma~\ref{lem:FPnBM} below.
This 
action 
is free, proper and smooth, so by the Quotient Manifold Theorem
\cite[Thm~21.10]{leeIntro} there exists a unique smooth structure on the 
projective spaces
$\mathbb{R}\mathbb{P}^n=\Sp^{n}/\Sp^0$, $\mathbb{C}\mathbb{P}^n=\Sp^{2n+1}/\Sp^1$ and $\mathbb{H}\mathbb{P}^n=\Sp^{4n+3}/\Sp^3$ that makes 
the quotient projection $\pi$ from a sphere onto $\mathbb{F}\mathbb{P}^n$ a smooth submersion i.e.
its differential $d\pi_x$ is surjective at each point $x$ in the sphere. 

The spheres 
$\Sp^n$, $\Sp^{2n+1}$ and  $\Sp^{4n+3}$
are Riemannian manifolds with the respective metrics induced by the ambient Euclidean spaces
and 
the groups of unit elements 
$\Sp^0,\Sp^1$ and $\Sp^3$ act via isometries. Thus, the pushforward Riemannian metric on 
$\mathbb{F}\mathbb{P}^n$ is well-defined and 
turns the projection $\pi$ into  a Riemannian submersion. Put differently, 
the sphere 
$\Sp^n$ (resp. $\Sp^{2n+1}$, $\Sp^{4n+3}$) is a smooth fibre bundle over $\mathbb{R}\mathbb{P}^n$
(resp. $\mathbb{C}\mathbb{P}^n$, $\mathbb{H}\mathbb{P}^n$)
with the fibre diffeomorphic to 
$\Sp^0$ (resp. $\Sp^1$, $\Sp^3$)
and the differential
$d\pi_x$ is an isometry when restricted to the space of horizontal tangent vectors at any $x$ in the sphere and
the tangent space at $\pi(x)$ in the projective space.  
This statement is trivial when $\mathbb{F}=\mathbb{R}$,
since in that case we have $\Sp^0=\{-1,1\}$ and $\pi$ is a local diffeomorphism. 
The other cases are given in~\cite[Problem 3-8]{leeRiem}. If
$\mathbb{F}=\mathbb{C}$ this construction yields the well-known Fubini-Study metric on
$\mathbb{C}\mathbb{P}^n$ \cite[Ch.~7.1]{jostRiem}. 

The BM on $\mathbb{F}\mathbb{P}^n$ can be defined as a strong Markov process with the generator
equal to 
$\frac{1}{2}\Delta_{\mathbb{F}\mathbb{P}^n}$, where $\Delta_{\mathbb{F}\mathbb{P}^n}$ is the Laplace-Beltrami operator corresponding to 
the Riemannian metric on  
$\mathbb{F}\mathbb{P}^n$ \cite[p.74]{hsumanifolds}.
Finally, we establish a connection between the spherical BM and the BM on projective spaces.  

\begin{lemma}
\label{lem:FPnBM}
Let $n\ge1$ and let $Z$ be BM on sphere $\Sp^{n}$ (resp. $\Sp^{2n+1}$,
$\Sp^{4n+3}$) started at $z$. Then the process $\pi(Z)$ is the BM on
$\mathbb{R}\mathbb{P}^n$ (resp. $\mathbb{C}\mathbb{P}^n$,
$\mathbb{H}\mathbb{P}^n$) started at $\pi(z).$ 
\end{lemma}

\begin{proof}
Fibres of the Riemannian submersion $\pi$ are orbits in
$\Sp^{n}\subseteq\mathbb{R}^{n+1}$ (resp.
$\Sp^{2n+1}\subseteq\mathbb{C}^{n+1}$, $\Sp^{4n+3}\subseteq\mathbb{H}^{n+1}$)  of
the (right) group action $((z_0,\ldots,z_n),\lambda)\mapsto(z_0\lambda ,\ldots,z_n\lambda)$, 
where $(z_0,\ldots,z_n)^\top\in \Sp^n$ (resp.
$\Sp^{2n+1}$, $\Sp^{4n+3}$) and $\lambda$ is an element of multiplicative group of
unit elements in $\mathbb{R}$ (resp. $\mathbb{C}$, $\mathbb{H}$). This group is
isomorphic to $\Sp^0$ (resp. $\Sp^1,\Sp^3$) and isometric to the submanifold 
$\{(z_0,\ldots,z_n)^\top\in \Sp^n \ \text{ (resp. $\Sp^{2n+1}$, $\Sp^{4n+3}$); $z_1=\cdots=z_n=0$}\}$ 
in $\Sp^n$ (resp. $\Sp^{2n+1}$, $\Sp^{4n+3}$).
This submanifold is easily
seen to be totally geodesic i.e. any geodesic in the submanifold is also
a geodesic in the ambient manifold. Moreover, each fibre (i.e. orbit) can be isometrically mapped onto
this submanifold
by multiplication from the left by
suitable element of special orthogonal group (for any element $x$ in the fibre
pick any matrix $O(x)$ in the special orthogonal group such that
$O(x)x=(1,0,\ldots,0)^\top$). 
Hence all fibres are totally geodesic submanifolds. 

Let $(M,g)$ be a $k$-dimensional Riemannian submanifold of a Riemannian manifold
$(\widetilde{M},\widetilde{g})$, i.e. the metric $g$ is a restriction of $\widetilde g$ to the tangent bundle of $M$. 
For any two vector fields $X,Y$ on $M$, 
the second fundamental form $II$ is given by
$II(X,Y)=\widetilde{\Delta}_{\widetilde{X}} \widetilde{Y}-\Delta_X Y$, where
$\Delta$ (resp. $\widetilde{\Delta}$) represent the metric connections on $M$
(resp. $\widetilde{M}$) and the vector fields $\widetilde{X},\widetilde{Y}$ are
arbitrary extensions of the vector fields $X,Y$ to $\widetilde{M}$.
The second fundamental form is well-defined and takes values in
the normal bundle of the submanifold $M$  (see \cite[Thm~8.2]{leeRiem}).
By~\cite[Exercise~8.4]{leeRiem}, the submanifold $M$ is totally geodesic if and only if 
$II(X,Y)=0$ 
for any vector fields $X,Y$ on $M$. 
The mean curvature $H_x$ at
any $x\in M$ 
is proportional to the metric trace of $II$, i.e.
for any  orthonormal frame
$e_1,\ldots,e_k$ in the neighbourhood of $x$ we have 
$H_x=\frac{1}{k}\sum_{i=1}^{k}II(e_i,e_i)_x$. Clearly, $H_x$ is equal to
$0$ for each $x\in M$ if $M$ is totally geodesic in $\widetilde M$.  

By~\cite[Thm~1]{pauwSubmer}, the projection $\pi(Z)$ in 
$\mathbb{R}\mathbb{P}^n$ (resp.
$\mathbb{C}\mathbb{P}^n,\mathbb{H}\mathbb{P}^n$) 
of the BM $Z$ on $\Sp^{n}$ (resp. $\Sp^{2n+1}$,
$\Sp^{4n+3}$) 
is a BM with the drift given by 
$V_x=-\frac{n}{2}d\pi_x(H_x)$ (resp. $-\frac{2n}{2}
d\pi_x(H_x),-\frac{4n}{2} d\pi_x(H_x)$), 
where $H_x$ is the mean curvature of the fibre of $x$ in the ambient sphere.
Since all the fibres are totally geodesic, the drift vanishes and the lemma follows. 
\end{proof}

\section*{Acknowledgements}
AM and VM are supported by The Alan Turing Institute under the EPSRC grant EP/N510129/1;
AM supported by EPSRC grant EP/P003818/1 and the Turing Fellowship funded by the Programme on Data-Centric Engineering of Lloyd's Register Foundation;
GUB supported by CoNaCyT grant FC-2016-1946 and UNAM-DGAPA-PAPIIT grant IN115217.

\bibliographystyle{amsalpha}
\bibliography{source}

\end{document}